\documentclass[11pt]{article}
\usepackage[misc]{ifsym} 
\usepackage{lipsum}
\usepackage{amssymb}
\usepackage{amsmath}
\usepackage{latexsym}
\usepackage{bm}
\usepackage{enumitem}
\usepackage{graphicx}
\usepackage{amscd}
\usepackage{extarrows}
\usepackage{amsfonts}
\usepackage{amssymb}
\usepackage{indentfirst}
\usepackage{bbm}
\usepackage{mathdots}
\usepackage{mathtools}
\usepackage[linktocpage]{hyperref}
\usepackage{amsthm}
\hypersetup{
    colorlinks,
    citecolor=black,
    filecolor=black,
    linkcolor=blue,
    urlcolor=black
}

\usepackage{quiver}
\usepackage[multiple]{footmisc}
\usepackage[title]{appendix}

\newcommand\blfootnote[1]{%
  \begingroup
  \renewcommand\thefootnote{}\footnote{#1}%
  \addtocounter{footnote}{-1}%
  \endgroup
}

\usepackage{float}

\theoremstyle{plain}
\newtheorem{theorem}{Theorem}[section]
\newtheorem{proposition}[theorem]{Proposition}
\newtheorem{lemma}[theorem]{Lemma}
\newtheorem{corollary}[theorem]{Corollary}

\theoremstyle{remark}

\theoremstyle{definition}
\newtheorem{definition}{Definition}[section]
\newtheorem{example}[theorem]{Example}

\DeclareMathOperator{\End}{End}

\DeclareMathOperator{\Hom}{Hom}

\DeclareMathOperator{\tr}{tr}

\DeclareMathOperator{\sh}{sh}

\DeclarePairedDelimiter\ceil{\lceil}{\rceil}\DeclarePairedDelimiter\floor{\lfloor}{\rfloor}

\title{A Level-Depth Correspondence between Verlinde Rings and Subfactors}

\author{Jun Yang\thanks{Harvard University, Cambridge, MA 02138, USA}}

\date{}
\begin{document}
\maketitle

\blfootnote{\Letter~\href{mailto:junyang@fas.harvard.edu}{junyang@fas.harvard.edu}}


\begin{abstract}
We establish a correspondence between the levels of Verlinde rings and the depths of subfactors. 
Given the $l$-level Verlinde ring $R_l(G)$ of a simple compact Lie group $G$, the tensor products of fundamental representations give us the inclusion of a pair of $\text{II}_1$ factors $N\subset M$. 
For the depth $d$ of $N\subset M$, we first prove $d=l$ for type $A_n,C_n$ and $B_2$.
More generally, the depth $d$ is shown to satisfy
\begin{center}
    $\beta\cdot l\leq d\leq l$ with $\beta\in (0,1)$,
\end{center}
where $\beta$ is uniquely determined by the simple type of $G$.   
We also show that the simple $N$-$N$-bimodules contained in $L^2(M)$ generate the Verlinde ring $R_l(G)$ as its fusion category.

\end{abstract}

\tableofcontents

\section{Introduction}

The Verlinde ring is a fusion category arising from the positive energy representation of loop groups \cite{PSloop} and also from the representation theory of quantum groups \cite{BKtensor}.  
There are several distinct approaches to the Verlinde ring, i.e., from the view of algebraic geometry by G. Faltings \cite{Falt94},  operator algebra by A. Wassermann \cite{Was98}, and twisted K-theory by D. Freed, M. Hopkins and C. Teleman \cite{FMT3}. 
It can be shown to be isomorphic to a quotient of the representation ring $R(G)$ of a compact simple Lie group $G$, or, equivalently, the representation ring $R(\mathfrak{g})$ of the corresponding simple complex Lie algebra $\mathfrak{g}$. 
The kernel of this quotient is uniquely determined by a positive integer $l$, which is usually called the {\it level}.

Generally speaking, a fusion category usually has a strong correspondence with subfactors, which denotes the inclusion of a pair of von Neumann algebras of type $\text{II}_1$ with trivial centers (factors). 
More precisely, we are always able to construct an inclusion pair of factors $N\subset M$ such that a certain tensor category within it is isomorphic to a given tensor category $\mathcal{C}$. 
T. Hayashi and S. Yamagami \cite{HY00afd} first realized all the $C^*$-tensor categories of bimodules over the hyperfinite $\text{II}_1$ factor. 
S. Falguières and S. Vaes \cite{FVa10cpt} showed the representation category of any compact group arises from the finite index bimodules of some $\text{II}_1$ factor. 
Then S. Falguières and S. Raum  \cite{FRcat13} treated the finite $C^*$-tensor category as well. 
Conversely, starting with a given tensor category, one can also generate subfactors. 
S. Sawin \cite{Saw95} first obtained subfactors from quantum groups with parameters that are not the roots of unity. 
For the case of roots of unity, H. Wenzl \cite{Wenz98} constructed subfactors from the tensor product of a module (and its dual) over quantum groups while F. Xu \cite{Xu98} constructed subfactors through quantum groups and the $\lambda$-lattices (see S. Popa's axiomatization \cite{Popa95}). 

For subfactors, there is a positive integer called the {\it depth}, denoted $d$,  which gives us much information about the inclusion.  
Given a subfactor $N\subset M$, one can iterate the basic construction, which plays a central role in the study of the index $[M:N]$ by V. Jones \cite{J83}. 
We then obtain a tower of $\text{II}_1$ factors: 
\begin{center}
$N\subset M=M_1\subset M_2=\langle M_1,e_1 \rangle \subset M_3=\langle M_2,e_2 \rangle \subset \cdots$,
\end{center}
where $e_k$ is the projection $L^2(M_{k})\to L^2(M_{k-1})$. 
The depth $d$ is then defined to be the minimal integer $k$ such that the center of the commutant $N'\cap M_k$ has its maximal dimension. 

This paper aims to give a subtle construction of subfactors from a Verlinde ring at level $l$ with the depth equal (or proportional) to $l$. 
Let $R_l(G)$ be the $l$-level Verlinde ring of a simple simply-connected compact Lie group $G$. 
\begin{theorem}
There is a subfactor $N\subset M$ constructed from the fundamental representations of $G$ with the depth $d(l)$ which satisfies
\begin{center}
$\beta\cdot l\leq d(l)\leq l$ with $\beta\in (0,1)$
\end{center}
for all simple types of $G$. 
Moreover, $d(l)=l$ for type $A_n$, $C_n$ and $B_2$. 
\end{theorem}

The motivation for this result originates from joint work \cite{JY21} of the author with V. Jones on Motzkin algebras, which can be constructed from $\End_G((V_0\oplus V_1)^{\otimes k})$ with $V_0,V_1$ the trivial and fundamental representation of $G=SU(2)$ or $U_q(\mathfrak{gl}_2)$ (see also \cite{BHMotz14}).  
Actually, their work contains an equivalent definition based on planar algebra \cite{Jpalg99}. 
For any level $l\geq 3$, they construct a subfactor $N\subset M$ of depth $l$ such that the bimodules generated from $_{N}L^2(M)_N$ are the $l$-level Verlinde ring of $SU(2)$, or equivalently, of type $A_1$. 

In this paper, we generalize this result to simple simply-connected compact Lie groups or simple complex Lie algebras. 
We also start with the $\mathfrak{g}$-module $V_0\oplus(\oplus_i V(\omega_i))$,  where $V_0$ is the trivial module and each $V(\omega_i)$ the is the irreducible representation with the fundamental weight $\omega_i$. 
It involves the Littlewood-Richardson problem in studying the decomposition of the tensor product of the highest weight representations (see \cite{Kum10}). 
As the trivial module is included, we obtain an increasing sequence of weights set which will finally contain all the weights $\lambda$ such that $(\lambda,\theta)\leq l$, which are the weights in $R_l(G)$ ($\theta$ is the highest root). 
The tower of the endomorphism algebras gives us a family of factors and also the commutants of the subfactors. 
The bimodules are then constructed in a canonical way from these commutants. 
We show the bimodules from $N\subset M$ have the same fusion rule as $R_l(G)$. 
\begin{corollary}
The tensor category generated by the $N$-$N$ bimodules in $L^2(M)$ is the Verlinde ring $R_l(G)$. 
\end{corollary}

In Section \ref{svld}, we have a brief review of the Verlinde rings.
In Section \ref{stenfund}, we construct the Verlinde ring from the direct sum of fundamental representations. 
In Section \ref{stower} and Section \ref{sfactor}, we construct the subfactors and describe the commutants. 
In Section \ref{sbimod1}, we construct a family of bimodules and describe their fusion rule. 

{\it Acknowledgements} 
The author is grateful for the comments from D. Bisch, A. Jaffe, and Y. Kawahigashi.  
This work was supported in part by the ARO Grant W911NF-19-1-0302 and the ARO MURI Grant W911NF-20-1-0082.

\section{The Verlinde Ring as a Quotient}\label{svld}

We first have a short review of some facts about complex semisimple Lie algebra. 
We mainly refer to \cite{Hum72}, for the basic Lie theory and to \cite{Beau93,Kum22} for the Verlinde rings. 

Let $G$ be a compact, simply-connected, simple Lie group and $\mathfrak{g}=\mathfrak{g}_{\mathbb{C}}$ be the complexified simple Lie algebra.
Let $\mathfrak{t}$ be the Cartan subalgebra of $\mathfrak{g}$. 
Denote the set of integral weights and the set of dominant integral weights by $P$ and $D$ respectively. 
Let $\Phi$ be the set of roots and $\Delta=\{\alpha_i,\dots,\alpha_n\}$ be the set of simple roots, where $n=\dim_{\mathbb{C}} \mathfrak{t}$.
Let $W=\langle s_{\alpha_1},\dots,s_{\alpha_n}\rangle$ be the Weyl group with each $s_{\alpha_i}$ the reflection given by the simple root $\alpha_i$. 

Let $\theta$ be the highest root and $\rho$ be the half-sum of positive roots. 
Let $(\cdot,\cdot)$ be the inner product on $\mathfrak{t}\cong \mathfrak{t}^*$ which is normalized in the sense $\|\theta\|^2=(\theta,\theta)=2$. 
Let $\alpha^{\vee}\colon=\frac{2\alpha}{(\alpha,\alpha)}$ be the coroot of $\alpha\in \Phi$.
Define $\langle\beta,\alpha\rangle=(\beta,\alpha^{\vee})=\frac{2(\beta,\alpha)}{(\alpha,\alpha)}$ for $\alpha,\beta \in \Phi$ (and also defined on $P$). 
Let $\omega_1,\dots,\omega_n$ be the fundamental weights, i.e., $\langle\omega_i,\alpha_j\rangle=\delta_{i,j}$ for all $1\leq i,j\leq n$.

Define $R(G)$ (or $R(\mathfrak{g})$) to be the representation ring of $G$ (or $\mathfrak{g}$). 
It is well-known that $R(G)\cong R(\mathfrak{g})=\mathbb{Z}[D]$, i.e., the $\mathbb{Z}$-linear span of the isomorphism classes of highest weight representations indexed by $D$.

Let $V(\lambda)$ be the irreducible representation with the highest weight $\lambda$, which will also stand for its isomorphism class in $R(\mathfrak{g})$. 
For a finite-dimensional $V$ representation of $\mathfrak{g}$, 
we let
\begin{equation*}
\begin{aligned}
    & \Pi(V)= \text{the set of all weights of } V;\\
    & \Pi_h(V)= \text{the set of all highest weights of the simple summands of } V.
\end{aligned}
\end{equation*}
For instance, if $V=\oplus_{\lambda\in D} m_{\lambda}\cdot V(\lambda)$ as the decomposition into irreducible representations, we have $\Pi_h(V)=\{\lambda\in D|m_{\lambda}\neq 0\}$. 
For each $1\leq i\leq n$, let $V(\omega_i)$ be the fundamental representation, which is the irreducible representation with the highest weight $\omega_i$.


Given an integer $l\geq 1$, we define
\begin{itemize}
 \item the {\it dominant integral weights at level $l$}
 \begin{center}
     $D_l=\{\lambda\in D|(\lambda,\theta)\leq l\}$,
 \end{center}
    \item the {\it affine wall} 
    \begin{center}
        $ H_{\alpha,m}=\{\lambda\in P|(\lambda,\alpha)=m(l+h^{\vee})\}$,
    \end{center}
    for $\alpha\in \Phi$, $m\in \mathbb{Z}$. 
    Let
    \begin{center}
        $H=\cup_{\alpha\in \Phi,m\in\mathbb{Z}}H_{\alpha,m}$. 
    \end{center}
    \item the {\it affine Weyl group at level} $l$
    \begin{center}
    $W_l=$ the group generated by $W$ and the map $\lambda\mapsto \lambda+(l+h^{\vee})\theta$.
    \end{center} 
    Note the action of $W_l$ on $P_{\mathbb{R}}=P\otimes_{\mathbb{Z}}\mathbb{R}$ is defined by $w*\lambda=w(\lambda+\rho)-\rho$ for $w\in W_l$ and $\lambda\in P_{\mathbb{R}}$. 
    We also define the set of minimal-length coset representatives in $W_l/W$ by $W_l'$. 
\end{itemize}

We define $I_l\subset R(\mathfrak{g})$ be the ideal spanned over $\mathbb{Z}$ by
\begin{enumerate}
    \item $V(\lambda)$ with $\lambda\in D$ and $\lambda+\rho\in H$,
    \item $V(w^{-1}*\mu)-\epsilon(w)V(\mu)$ with $\mu\in D_l$ and $w\in W_l'$. 
\end{enumerate}
The {\it Verlinde ring at level $l$} of $G$ (or $\mathfrak{g}$) is defined to be the quotient ring
\begin{center}
    $R_l(G)=R(G)/I_l$ (or $R_l(\mathfrak{g})=R(\mathfrak{g})/I_l$). 
\end{center}
We will denote the image of the isomorphism class of $V(\lambda)$ in the quotient ring by $[V(\lambda)]$.
We denote the quotient map by $\pi_l$ and 
the multiplication (tensor product) in $R_l(\mathfrak{g})$ by $\otimes_l$. 

The following result is well-known (see \cite{Beau93}, \cite{Kum22} Chapter 4 and \cite{FJKLM04} Chapter 2.3). 
\begin{proposition}\label{plfus}
\begin{enumerate}
    \item $R_l(\mathfrak{g})$ has a $\mathbb{Z}$-basis $\{[V(\lambda)]|\lambda\in D_l\}$;
    \item $\pi_l(V(\lambda))=[V(\lambda)]$ for $\lambda\in D_l$; 
    \item $[V(\lambda)]\otimes_{l}[V(\mu)]=[V(\lambda)\otimes V(\mu)]$ if $\lambda+\mu\in D_l$. 
\end{enumerate}
\end{proposition}

Indeed, these $[V(\lambda)]$ gives the family of positive energy representations of the loop group $LG=C^{\infty}(S^1,G)$ at level $l$ (see \cite{PSloop}). 
In the following sections, we will treat them as $LG$-modules. 
We will use the same notations as above for the weights and representations of $LG$ if there is no confusion.
For instance, $\Pi_h([V])$ will denote the highest weights of the irreducible $LG$-modules in the decomposition of a $LG$-module $[V]$.

\begin{proposition}\label{pdlwts}
Suppose $\lambda_1,\cdots,\lambda_t\in D$ such that $\sum_{1\leq i\leq t}\lambda_i\in D_l$. 
We have $\Pi_h(\otimes_{1\leq i\leq t}V(\lambda_i))\subset D_l$. 
\end{proposition}
\begin{proof}
Note any weight in $\Pi_h(\otimes_{1\leq i\leq t}V(\lambda_i))$ must be of the form $\sum_{1\leq i\leq t}\lambda_i-\sum_{1\leq j\leq n}y_i\cdot \alpha_i$ with each $y_i\in \mathbb{Z}_{\geq 0}$.  
It suffices to show
\begin{center}
    $(\left(\sum_{1\leq i\leq t}\lambda_i-\sum_{1\leq j\leq n}y_i\cdot \alpha_i\right),\theta)\leq l$. 
\end{center}
As $\sum_{1\leq i\leq t}\lambda_i\in D_l$, we have $(\sum_{1\leq i\leq t}\lambda_i,\theta)\leq l$. 
It then suffices to show $(\alpha_j,\theta)\geq 0$ for each $1\leq j\leq n$, which follows the fact that $\theta\in D$. 
\end{proof}

\section{Tensor Products of Fundamental Representations}\label{stenfund}

In this section, as $\mathfrak{g}$-modules, we consider how the fundamental representations  generate the irreducible representations of level $l$, which are the ones with highest weights in $D_l=\{\lambda\in D|(\lambda,\theta)\leq l\}$. 
Then we move to the case of $LG$-modules and the Verlinde ring. 

Consider the $\mathfrak{g}$-module: 
\begin{center}
    $W=V(0)\oplus(\oplus_{1\leq i\leq n}V(\omega_i))$,
\end{center}
which is the direct sum of trivial module $V(0)=\mathbb{C}$ and all fundamental representations $V(\omega_i)$'s. 
We have the following increasing sequence of sets of dominant weights:
\begin{center}
$\Pi_h(W^{\otimes 0})\subset \Pi_h(W^{\otimes 1})\subset \Pi_h(W^{\otimes 2})\subset \cdots\subset\Pi_h(W^{\otimes k})\subset \Pi_h(W^{\otimes k+1})\subset\cdots$,
\end{center}
where $\Pi_h(W^{\otimes 0})=\{0\}$ and $\Pi_h(W^{\otimes 1})=\{0,\omega_1,\cdots,\omega_n\}$ by the definition. 
Observe $D_l$ is a finite set. 
By Proposition \ref{plfus} and the fact that fundamental representations generate $R(\mathfrak{g})$, we know there exists some $d(l)$ depending on the simple type of the Lie algebra $\mathfrak{g}$ such that
\begin{center}
  $d(l)=d_{\mathfrak{g}}(l)=\min\{k\geq 0|D_l\subset \Pi_h(W^{\otimes k})\}$.  
\end{center}
This is equivalent to say: 
\begin{lemma}\label{liffdl}
For each $l\geq 0$, there is an integer $d(l)\geq 0$ such that
\begin{center}
$D_l\subset \Pi_h(W^{\otimes k})$ if and only if $k\geq d(l)$. 
\end{center}
\end{lemma}
Then we pass to $LG$-modules at level $l$, where their highest weights are always contained in $D_l$.  
We will prove (see Corollary \ref{cdldl})
\begin{center}
    $d(l)=d_{\mathfrak{g}}(l)=\min\{k\geq 0|D_l= \Pi_h([W]^{\otimes k})\}$.
\end{center}
The rest of this section is mainly devoted to the following result:

\begin{theorem}\label{tdl}
\begin{enumerate}
    \item  For type $A_n$, $C_n$ or $B_2$, $d(l)=l$;
    \item For type $B_n$ ($n\geq 3$), $\ceil{\frac{2l}{n}}\leq d(l)\leq l$;
    \item For type $D_n$ ($n\geq 4$), $\ceil{\frac{2l}{n-1}}\leq d(l)\leq l$;
    \item For type $E_6$, $E_7$ or $E_8$, $\ceil{\frac{l}{3}}\leq d(l)\leq l$, $\ceil{\frac{l}{5}}\leq d(l)\leq l$, $\ceil{\frac{4l}{15}}\leq d(l)\leq \floor{\frac{l}{2}}$ respectively;
    \item For type $F_4$, $\ceil{\frac{2l}{5}}\leq d(l)\leq l$;
    \item For type $G_2$, $\ceil{\frac{2l}{3}}\leq d(l)\leq l$. 
\end{enumerate}
\end{theorem}

We first consider the map $\varepsilon\colon P\to \mathbb{Z}$ given by
\begin{center}
$\varepsilon(\sum_{1\leq i\leq n}  x_{i}\omega_i)=\sum_{1\leq i\leq n} x_i$. 
\end{center}
For each $k\geq 0$, we define a set of dominant weights
\begin{center}
    $B_k=\{\lambda=\sum_{1\leq i\leq n}  x_{i}\omega_i|\sum_{1\leq i\leq n} x_i\leq k,x_i\in \mathbb{Z}_{\geq 0}\}$,
\end{center} 
or, equivalently, $B_k=D\cap \varepsilon^{-1}([0,k])$. 
\begin{example}
For the group $SU(2)$ (type $A_1$), $W=V_0\oplus V(\omega_1)$,  $\varepsilon(\Pi_h(W^{\otimes k}))=\{0,1,\cdots,k\}$ by the Clebsch–Gordan formula. 
We can further show $D_k=\Pi_h(W^{\otimes k})=B_k$.  

As shown in \cite{GPE8} (see page 351), for $E_8$ and its fundamental representation $\omega_5$, $V(\omega_5)\otimes V(\omega_5)$ contains $V(5\omega_1+\omega_7)$. 
Hence $\varepsilon(5\omega_1+\omega_7)=6$ and $6\in\varepsilon(\Pi_h(V(\omega_5)\otimes V(\omega_5)))\subset  \varepsilon({\Pi_h(W^{\otimes 2})})$.
So $\Pi_h(W^{\otimes k})$ may be strictly larger than $B_k$.  
\end{example}

\begin{proposition}\label{pBkwt}
For each simple complex Lie algebra $\mathfrak{g}$, we have $B_k\subset \Pi_h(W^{\otimes k})$. 
For $k\geq 1$, we further obtain
\begin{enumerate}
    \item  For type $A_n$, $C_n$ or $B_2$, $\Pi_h(W^{\otimes k})=B_k$;
    \item For type $B_n$ ($n\geq 3$), $B_k\subset \Pi_h(W^{\otimes k})\subset B_{\floor{\frac{nk}{2}}}$;
    \item For type $D_n$ ($n\geq 4$), $B_k\subset \Pi_h(W^{\otimes k})\subset B_{\floor{\frac{(n-1)k}{2}}}$;
    \item For type $E_6$, $E_7$ or $E_8$, $B_k\subset \Pi_h(W^{\otimes k})\subset B_{3k}$, $B_{5k}$ or $B_{\floor{\frac{15k}{2}}}$ respectively;
    \item For type $F_4$, $B_k\subset \Pi_h(W^{\otimes k})\subset B_{\floor{\frac{5k}{2}}}$;
    \item For type $G_2$, $B_k\subset \Pi_h(W^{\otimes k})\subset B_{\floor{\frac{3k}{2}}}$. 
\end{enumerate}
\end{proposition}
\begin{proof}

Let us first prove $B_k\subset \Pi_h(W^{\otimes k})$ by induction. 
It is straightforward to check $B_1\subset \Pi_h(W)$. 
We take $\lambda=\sum_{1\leq i\leq n}  x_{i}\omega_i\in \Pi_h(W^{\otimes k})$ such that $\sum_{1\leq i\leq n} x_i=k$. 
Consider the tensor product $V(\lambda)\otimes V(\omega_j)$. 
It has a simple summand with the highest weight $\omega_j+\sum_{1\leq i\leq n}  x_{i}\omega_i$, which is in $B_{k+1}$. 
This shows that $\Pi_h(W^{\otimes k+1})$ contains all the weights of the form $\sum_{1\leq i\leq n}  x_{i}\omega_i$ with $\sum_{1\leq i\leq n} x_i=k+1$. 

Meanwhile, $W^{\otimes k} $ is a proper subspace of $ W^{\otimes k+1}$ as $W$ contains the trivial representation, which is to say $B_k\subset \Pi_h(W^{\otimes k})\subset\Pi_h(W^{\otimes k+1})$. 
Hence $B_{k+1}\subset \Pi_h(W^{\otimes k+1})$.

Now we describe an upper bound of $\Pi_h(W^{\otimes k})$. 
Observe $\Pi_h(W^{\otimes k})$ consists the elements of the form 
\begin{center}
 $\mu=\sum_{1\leq i\leq n}x_i\omega_i-\sum_{1\leq i\leq n}y_i\alpha_i$, with $\sum x_i\leq k$ and $x_i,y_i\geq 0$,  
\end{center}
which subjects to the  conditions $\langle \mu,\alpha_j\rangle\geq 0$ for all $1\leq j\leq n$. 
This is equivalent to the linear inequalities 
\begin{center}
 $\vec{y}\cdot A\leq \vec{x}$,   
\end{center}
where $\vec{x}=(x_1,\cdots,x_n)$, $\vec{y}=(y_1,\cdots,y_n)$ and $A=[\langle \alpha_i,\alpha_j\rangle]_{n\times n}$ is the Cartan matrix of $\mathfrak{g}$. 
Hence we have
\begin{enumerate}
    \item $\sum_{1\leq j\leq n}y_j\langle \alpha_j,\alpha_i\rangle\leq x_i$ for $1\leq i\leq n$ (by $\mu\in D$);
    \item $\varepsilon(\sum_{1\leq i\leq n}y_i\alpha_i)\leq \sum_{i} x_i\leq k$ (by $\mu\in D$),
    \item $y_i \geq 0$ for $1\leq i\leq n$ (by $\sum_{1\leq i\leq n}x_i \omega_i$ are highest). 
\end{enumerate}
Please note $\sum_{1\leq j\leq n}\langle\alpha_i,\alpha_j\rangle \geq 0$ for each $i$ in type $A_n$, $B_2$ or $C_n$,. 
Hence each $-y_{i}\alpha_i$ contributes $-y_{i}\sum_{1\leq j\leq j}\langle\alpha_i,\alpha_j\rangle $ to $\varepsilon(\mu)$, which is non-positive. 
We have $\varepsilon(\mu)\leq \sum_{1\leq i\leq n}x_i$ and  $\Pi_h(W^{\otimes k})=B_k$. 

For the remaining types, we apply the simplex method \cite{MurLP83} and induction on the rank $n$ to get the maximal values of $\varepsilon(\mu)$. 
We leave it to the reader to check the linear inequalities. 
\end{proof}

Assume $\theta=\sum_{1\leq i\leq n}c_i\cdot \alpha_i$ as a $\mathbb{Z}$-linear combination of simple roots.  

\begin{lemma}\label{lbk}

If $\min_{1\leq i\leq n}\{\frac{c_i\|\alpha_i\|^2}{2}\}=c$, we have $D_l\subset B_k$ if and only if $k\geq \floor{\frac{l}{c}}$.

\end{lemma}
\begin{proof}
Without loss of generality, we assume $c=1$. 
Let $\lambda=\sum_{1\leq i\leq n}  n_{i}\omega_i$ and observe 
\begin{equation*}
\begin{aligned}
(\lambda,\theta)&=\sum_{1\leq i\leq n}c_i(\lambda,\alpha_i)=\sum_{1\leq i\leq n}c_i\cdot \frac{2(\lambda,\alpha_i)}{(\alpha_i,\alpha_i)}\frac{(\alpha_i,\alpha_i)}{2}\\
&=\sum_{1\leq i\leq n}c_i\cdot \frac{(\alpha_i,\alpha_i)}{2}\langle \lambda,\alpha_i\rangle=\sum_{1\leq i\leq n}\frac{c_i\|\alpha_i\|^2}{2}\cdot x_i.
\end{aligned}
\end{equation*}
Hence $D_l=\{\lambda=\sum_{1\leq i\leq n}  x_{i}\omega_i|\sum_{1\leq i\leq n}\frac{c_i\|\alpha_i\|^2}{2}\cdot x_i\leq l\}$. 

The inclusion $D_l\subset B_l$ is straightforward by Lemma \ref{lbk} as all $\frac{c_i\|\alpha_i\|^2}{2}\geq 1$.  
As $B_k\subset B_{k+1}$, it suffices to show $D_l\not\subset B_{l-1}$. 
Suppose $\frac{c_j\|\alpha_j\|^2}{2}=1$ for some $j$. 
Then $l\cdot\omega_j\in D_l$ but $l\cdot\omega_j\notin B_{l-1}$. 

For $c\geq 1$, $D_l=B_{\floor{\frac{l}{c}}}$ is clear. 
Suppose $\frac{c_j\|\alpha_j\|^2}{2}=c$ for some $j$.  
We have $\floor{\frac{l}{c}}\cdot \omega_j\in D_l$ but $\floor{\frac{l}{c}}\cdot \omega_j\notin B_{\floor{\frac{l}{c}}-1}$.
\end{proof}

We now consider the simple types of $\mathfrak{g}$ for the construction above. 
We refer to \cite{BbLie2} Chapter VI.4 for notations and more details. 
\begin{proposition}\label{ptypelie}
Assume the highest root $\theta=\sum_{1\leq i\leq n}c_i\cdot \alpha_i$. 
Then we have
\[
    \min_{1\leq i\leq n}\{\frac{c_i\|\alpha_i\|^2}{2}\}= 
\begin{cases}
    1,& \text{if } \mathfrak{g} \text{ is of type } A_n,B_n,C_n,D_n,E_6,E_7,F_4\text{ or } G_2\\
    2,              & \text{if } \mathfrak{g} \text{ is of type } E_8.
\end{cases}
\]
\end{proposition}
\begin{proof}
\begin{enumerate}
    \item {\it Type $A_n$:}
    We have all $c_i=1$ as $\theta=\sum_{1\leq i\leq n}\alpha_i=\varepsilon_1-\varepsilon_{n+1}$.  
    The Euclidean inner product is normalized in the sense of $\|\theta\|^2=2$. 
    Hence $\frac{c_i\|\alpha_i\|^2}{2}=1$ for each $i$.

    \item {\it Type $B_n$:}
    We have all $c_1=1$ and $c_i=2$ for $2\leq i\leq n$ as $\theta=\alpha_1+2\alpha_2+\cdots+2\alpha_n=\varepsilon_1+\varepsilon_{2}$.  
    The Euclidean inner product is normalized in the sense of $\|\theta\|^2=2$. 
    Hence $\min_{1\leq i\leq n}\{\frac{c_i\|\alpha_i\|^2}{2}\}=\frac{c_1\|\alpha_1\|^2}{2}=1$. 

     \item {\it Type $C_n$:}
    We have all $c_n=1$ and $c_i=2$ for $1\leq i\leq n-1$ as $\theta=2\alpha_1+\cdots+2\alpha_{n-1}+\alpha_n=2\varepsilon_1$.  
    The normalized inner product is one-half of the Euclidean one. 
    So $\|\alpha_i\|^2=1$ for $1\leq i\leq n-1$ and $\|\alpha_n\|^2=2$. 
    Hence $\frac{c_i\|\alpha_i\|^2}{2}=1$ for each $i$.

    \item {\it Type $D_n$:}
    We have all $c_1=c_{n-1}=c_n=1$ and $c_i=2$ for $2\leq i\leq n-2$ as $\theta=\alpha_1+2\alpha_2+\cdots+2\alpha_{n-2}+\alpha_{n-1}+\alpha_n=\varepsilon_1+\varepsilon_2$.
    The Euclidean inner product is normalized in the sense of $\|\theta\|^2=2$ and we have $\|\alpha_i\|^2=2$ for each $i$. 
    Hence $\frac{c_i\|\alpha_i\|^2}{2}=1$ when $i=1,n-1,n$ are the minimal values . 
\end{enumerate}
Now let us consider the exceptional types. 
Note the $E_6,E_7$ type can be embedded into $E_8$ as subsystems. 
\begin{enumerate}
  \setcounter{enumi}{4}
  \item {\it Type $E_8$:}
$\theta=2\alpha_1+3\alpha_2+4\alpha_3+6\alpha_4+5\alpha_5+4\alpha_6+3\alpha_7+2\alpha_8=\varepsilon_1+\varepsilon_8$. 
  We have all $\|\alpha_i\|^2=2$. 
  Hence $\frac{c_i\|\alpha_i\|^2}{2}=2$ are minimal with value when $i=1,8$. 
  \item {\it Type $E_7$:}
$\theta=2\alpha_1+2\alpha_2+3\alpha_3+4\alpha_4+3\alpha_5+2\alpha_6+\alpha_7=\varepsilon_8-\varepsilon_7$.  
  We have  all $\|\alpha_i\|^2=2$. 
  Hence $\frac{c_i\|\alpha_i\|^2}{2}=1$ are minimal with value  when $i=7$.
  \item {\it Type $E_6$:}
$\theta=\alpha_1+2\alpha_2+2\alpha_3+3\alpha_4+2\alpha_5+\alpha_6=1/2(\varepsilon_1+\cdots+\varepsilon_5-\varepsilon_6-\varepsilon_7+\varepsilon_8)$.  
  We have all $\|\alpha_i\|^2=2$. 
  Hence $\frac{c_i\|\alpha_i\|^2}{2}=1$ are minimal with value  when $i=1,6$.

  \item {\it Type $F_4$:}
$\theta=2\alpha_1+3\alpha_2+4\alpha_3+2\alpha_4=\varepsilon_1+\varepsilon_2$.  
  We have  $\|\alpha_1\|^2=\|\alpha_2\|^2=2$ and $\|\alpha_3\|^2=\|\alpha_4\|^2=1$. 
  Hence $\frac{c_i\|\alpha_i\|^2}{2}=1$ is minimal with value when $i=4$. 

  \item {\it Type $G_2$:}
$\theta=3\alpha_1+2\alpha_2=-\varepsilon_1-\varepsilon_2+2\varepsilon_3$. 
The normalized inner product is $1/3$ of the Euclidean one. 
  We have  $\|\alpha_1\|^2=2/3$ and $\|\alpha_2\|^2=2$. 
  Hence $\frac{c_i\|\alpha_i\|^2}{2}=1$ is minimal with value  when $i=1$.  
\end{enumerate}
\end{proof}

\begin{proposition}\label{pdlup}
For any simple complex Lie algebra $\mathfrak{g}$, $d(l)\leq l$.  

In particular, if $\mathfrak{g}$ is of type $E_8$, $d(l)\leq \floor{l/2}$. 
\end{proposition}

\begin{proof}
By Proposition \ref{ptypelie} and Lemma \ref{lbk}, we know that $c=1$ for all the types except $E_8$ (for which $c=2$). 
It then follows Lemma \ref{pBkwt}. 
\end{proof}

Now we pass to the Verlinde ring $R_l(\mathfrak{g})$ or, equivalently, the category of $LG$-modules. 
\begin{corollary}\label{cdldl}
    $d(l)=\min\{k\geq 0|D_l= \Pi_h([W]^{\otimes k})\}$.
\end{corollary}
\begin{proof}
By Proposition \ref{pdlwts}, we know $\Pi_h(W^k)\subset D_l$ when $B_k\subset D_l$. 
Then, by 2 of Proposition \ref{plfus}, we conclude
\begin{center}
$\Hom_{\mathfrak{g}}(V_\mu,W^{\otimes k})=\Hom_{LG}([V_\mu],[W]^{\otimes k})$, 
\end{center}
for all $\mu\in D_l$. 
This shows $W^{\otimes k}$ contains $V(\mu)$ if and only if $[W]^{\otimes k}$ contains $[V(\mu)]$ for each $\mu\in D_l$. 
\end{proof}

\begin{proof}[Proof of Theorem \ref{tdl}]
The upper bound is given in Proposition \ref{pdlup}. 

For type $A_n$, $C_n$ or $B_2$, it follows the fact $B_k=\Pi_h(W^{\otimes k})$ (see 1 of Proposition \ref{pBkwt}), Lemma \ref{lbk} and Proposition \ref{ptypelie}. 

For the remaining types except $E_8$, let us assume $d(l)\geq s$. 
Let $\beta$ be the value given in Proposition \ref{pBkwt}, i.e, $\beta=\frac{n}{2}$ for $B_n$,$\frac{n-1}{2}$ for $C_n$, $3,5$ for $E_6,E_7$, $\frac{5}{2}$ for $F_4$ or $\frac{3}{2}$ for $G_2$. 
We have
\begin{equation*}
\begin{aligned}
d(l)> k & \Leftrightarrow \text{There exists } \exists \lambda\in D_l, \lambda\notin \Pi_h(W^{\otimes k})&\\
&\Leftarrow  \exists \lambda\in D_l, \lambda\notin B_{\floor{\beta k}} &\\
&\Leftarrow \floor{\beta k}\leq l-1 &\\
&\Leftrightarrow \beta k<l\Leftrightarrow k<\beta^{-1}l.& 
\end{aligned}
\end{equation*}
Here we apply Proposition \ref{pBkwt} and Lemma \ref{lbk} in the second and third lines respectively. 
Then we conclude $d(l)\geq \beta^{-1}l$ or, equivalently $d(l)\geq \ceil{\beta^{-1}l}$. 

The inequality of $d(l)$ for $E_8$ follows similarly by Corollary \ref{lbk} and Proposition \ref{pBkwt}. 
\end{proof}

\section{Towers of Finite-Dimensional Algebras and Subfactors}\label{stower}

From this section, we fix the positive integer $l$and suppose $|D_l|=m$. 
Let $V_1,\dots,V_m$ denote the simple $LG$-modules in the Verlinde ring $R_l(\mathfrak{G})$.  

For each $k\geq 0$, define a finite-dimensional $C^*$-algebra
\begin{center}
$A_k=\End(W^{\otimes k})=\Hom(W^{\otimes k},W^{\otimes k})$, 
\end{center}
where we let $A_0=\mathbb{C}$.  
Observe $\Hom(V_i,V_j)=\mathbb{C}\delta_{i,j}$ and $\dim Z(A_k)$ is the number of isomorphism classes of simple modules contained in $W^{\otimes k}$. 
By Proposition \ref{ptypelie}, we know $\dim Z(A_k)=m$ when $k\geq l$ for all the types except $E_8$, or, $k\geq \floor{l/2}$ 
for type $E_8$. 

\underline{The left inclusion $i_k\colon A_{k}\to A_{k+1}$ }
There is a natural inclusion $i_k:A_k\hookrightarrow A_{k+1}$  defined as $i_k(f)=f\otimes \text{id}_W$. 
We denote the inclusion matrix of the pair $A_k\subset A_{k+1}$ by $T(k)=[t(k)_{i,j}]\in M_{m\times m}(\mathbb{Z})$, which is given by
\begin{center}
    $ t(k)_{i,j}=\dim_{\mathbb{C}}\Hom(V_i\otimes W,V_j)$. 
\end{center}
\begin{lemma}\label{lsymir}
For $k\geq d(l)$, the inclusion matrices are identical, i.e. $T_k=T$ for $k\geq d(l)$.  
Moreover, $T$ is symmetric and irreducible.  
\end{lemma}
\begin{proof} 
We first claim $W$ is self-dual.  
It is well-known that the dual of a simple $\mathfrak{g}$-module $V(\lambda)$ is given by $V(-w_{0}(\lambda))$, where $w_0$ is the longest element in the Weyl group $W$. 

Note $w_0$ sends the positive Weyl chamber to the negative one. 
Observe $-w_0^{-1}(\alpha_i)$ is still a simple root and $-w_0^{-1}$ acts as a permutation of $\Delta$, say $-w_0^{-1}(\alpha_i)=\alpha_{\sigma(i)}$ for some $\sigma \in S_n$.  
We have
\begin{equation*}
\begin{aligned}
\langle -w_0^{-1}(\lambda_i),\alpha_j\rangle&=\frac{2(-w_0^{-1}(\lambda_i),\alpha_j)}{(\alpha_j,\alpha_j)}=\frac{2(-\lambda_i,w_0^{-1}(\alpha_j))}{(w_0^{-1}(\alpha_j),w_0^{-1}(\alpha_j))}\\
&=\frac{\lambda_i,\alpha_{\sigma(j)})}{(\alpha_{\sigma(j)},\alpha_{\sigma(j)})}=\delta_{i,{\sigma(j)}}. 
\end{aligned}
\end{equation*}
Hence $-w_0^{-1}(\lambda_i)=\lambda_{\sigma(i)}$ and $W$ is self-dual. 

Thus we obtain
\begin{equation*}
\begin{aligned}
t(k)_{i,j}&=\dim\Hom(V_{i}\otimes W,V_{j})=\dim\Hom(V_{i},V_{j})\otimes W^{*})\\&=\dim\Hom(V_{i},V_{j}\otimes W)=\dim\Hom(V_{j}\otimes W,V_{i})=t(k)_{j,i}, 
\end{aligned}
\end{equation*}
which is independent with $k$ by Proposition \ref{pdlup} once $k\geq d(l)$. 
Hence $T(k)=T(k)^{t}=T$ for some  $T$ if $k\geq d(l)$. 

For the irreducibility of $T$, it suffices to show the associated graph is strongly connected.
This is equivalent to $\sum_{s=1}^{S}T^s$ is positive for sufficiently large $S$. 
Suppose $T^s=[t_{i,j}^{(s)}]$ and fix a pair of indices $(i,j)$. 
There exist positive integers $a,b$ such that $W^{\otimes a},W^{\otimes b}$ have the summands $V_{i}$, $V_{j}$ respectively. 
Let $s=a+b$ and we obtain
\begin{equation*}
\begin{aligned}
t^{(s)}_{i,j}&=\dim\Hom(V_{i}\otimes W^{\otimes a+b},V_{j})=\dim\Hom(V_{i}\otimes W^{\otimes b},V_{j}\otimes  W^{\otimes a})\\&\geq \dim\Hom(V_{i}\otimes V_{j},V_{j}\otimes V_{i})\geq 1. 
\end{aligned}
\end{equation*}
Hence the associated graph is strongly connected. 
\end{proof}

\begin{proposition}\label{ptr}
The algebra $\cup_{k\geq 0}A_k$ admits a unique tracial state. 
\end{proposition}
\begin{proof}
By the Perron-Frobenius theorem, the inclusion matrix $T=T_{k}$ ($k\geq d+1$) admits an eigenvalue $\beta\in \mathbb{R}_{+}$ such that $|\beta|$ is strictly greater than the others. 
Its eigenvector $V_\beta$ has all its components positive. 
As $T$ is irreducible by Lemma \ref{lsymir}, 
one can show the space of tracial states is a singleton and hence contains a factor trace (see \cite{Tak3} Chapter XIX, Lemma 3.9). 
\end{proof}


This trace will yield the hyperfinite $\text{II}_1$ factor as its completion in the GNS construction. 
We denote this hyperfinite $\text{II}_1$ factor by $M$ and the trace by $\tr$. 

\underline{The conditional expectation $E_{k+1}$} For each $A_k$, we consider its completion with respect to $tr$, which is Hilbert space and will be denoted as $L^2(A_k,\tr)$. 
Let $e_{k+1}:L^2(A_k,\tr)\to L^2(A_{k-1},\tr)$ be the orthogonal projection, which is comes from the embedding $i_{k-1}$. 
The projection $e_{k+1}$ will certainly induce a map $E_{k+1}:A_k\to A_{k-1}$ called the {\it conditional expectation}. 
Consider the action of $A_k$ and $e_{k+1}$ on $L^2(A_k)$.
They generate a von Neumann $(A_k\cup\{e_{k+1}\})''$, denoted $\langle A_{k},e_{k+1}\rangle$. 
This is the {\it basic construction} of finite-dimensional $C^*$-algebras.

\begin{lemma}\label{lbasicon}
We have $\langle A_{k},e_{k+1}\rangle\subset A_{k+1}$. 
If $k\geq d+1$, $\langle A_{k},e_{k+1}\rangle=A_{k+1}$. 

\end{lemma}
\begin{proof} 
Note the inclusion matrix $T^{\langle A_{k},e_{k+1}\rangle}_{A_k}=(T^{A_k}_{A_{k-1}})^{\rm t}=T_{k-1}^{\rm t}$ for $A_k\subset \langle A_{k},e_{k+1}\rangle$. 
It suffices to show $T_k-T_{k-1}^{\rm t}$ is positive in general and $T^{\langle A_{k},e_{k+1}\rangle}_{A_k}=T_{k-1}^{\rm t}=T_{k-1}$ if $k\geq d+1$. 
     
Note that $W$ contains the trivial representation $V_0$ as a summand. 
Hence the number of any irreducible object $V_{i}$ at depth $k$ is no greater than that at depth $k+1$.  
So $t(k+1)_{i,j}\geq t(k)_{j,i}$ or equivalently $T_{k+1}-T_{k}^{\rm t}$ is positive,  which implies $A_{k+1}$ always contains the algebra $\langle A_{k},e_k\rangle$. 

Observe $T_{k+1}=T_k={T_k}^{\rm t}$ is symmetric when $k\geq d+1$. 
By \cite{J83} Lemma 4.4.1, $A_{k+1}$ is the basic construction of the pair $A_{k-1}\subset^{e_k} A_k$.
\end{proof}

\underline{The right inclusion $i_{k,j+k}\colon A_k\to A_{j+k}$ }
There is another natural inclusion $i_{k,j+k}\colon A_{k}\subset A_{j+k}$ defined by
\begin{center}
$i_{k,j+k}(f)=\text{id}_{W^{\otimes j}}\otimes f$ 
\end{center}
for $j\geq 0$, which is in $\End(W^{\otimes (j+k)})=A_{j+k}$ for $f\in \End(W^{\otimes k})=A_{k}$. 
Thus it induces an inclusion $i_{j}^R:\cup_{k\geq 0}{A_k}\subset \cup_{k\geq 0}A_{j+k}$ (here $R$ denotes the inclusion on the right side). 
Indeed, this inclusion is a composite of $i_{k,k+1},i_{k+1,k+2},\dots,i_{k+j-1,k+j}$ and can be shown to be trace-preserving. 

Now let us consider the inclusion $i_j^R$
which maps the triple $A_{k-1}\subset A_k\subset \langle A_k,e_{k+1}\rangle$
to $A_{j+k-1}\subset A_{j+k}\subset A_{j+k+1}$. 
\begin{corollary}\label{ciee}
Within $B(L^2(A_{j+k}))$, we have $i_j^R(e_{k+1})=e_{j+k+1}$. 
\end{corollary}
\begin{proof}
Consider the restriction on the subspace $L^2(i_j^R(A_k))\subset L^2(A_{j+k})$, we  have $i_j^R(e_{k+1})$ is the orthogonal projection from $L^2(i_j^R(A_k))$ to 
 $L^2(i_j^R(A_{k-1}))$. 

Moreover, for $x\in  i_j^R(A_{k-1})$, we have $[i_j^R(e_{k+1}),x]=0$. 
It is clear that $i_j^R(e_{k+1})$ commutes with the elements in $A_{j+k-1}$.
This implies $i_j^R(e_{k+1})$ acts as the same as $e_{j+k}$ on $L^2(A_{j+k})$, which is the unique projection. 
\end{proof}

\section{The Commutants}\label{sfactor}

Consider the complex algebra $\cup_{k\geq 0}A_{j+k}$. 
By Proposition \ref{ptr}, $\cup_{k\geq 0}A_{j+k}$ also admits a factor trace. 
The GNS construction gives us a hyperfinite $\text{II}_1$ factor, denoted $M_j$.
Moreover, as $A_{k}\subset A_{j+k}$ for each $k\geq 0$,
$M$ is a subfactor of $M_j$. 
Thus we get an increasing tower of factors
\begin{center}
$M=M_0\subset M_1\subset M_2\subset M_3\subset \dots$. 
\end{center}
The commutants $M'\cap M_j$ will be discussed with commuting squares. 
We refer to \cite{JS97} for some basic facts about commuting squares and their properties. 
Now we consider the following diagram
\begin{equation*}
\begin{aligned}
&A_{j+k}~~&\subset^{i_{j+k}}~~~~&A_{j+k+1}\\
&\cup^{i_{k,j+k}}~~   &&\cup^{i_{k+1,j+k+1}}~~\\
&A_{k}~~&\subset^{i_k}~~~~&A_{k+1}
\end{aligned}
\end{equation*}
with $j\geq 0$. 
Please note the horizontal embeddings are the left inclusions while the vertical ones are the right inclusions. 

\begin{lemma}
We have $E_{j+k+2}(i_{k+1,j+k+1}(A_{k+1}))=A_{k}$. 
Hence the diagram above is a commuting square. 
\end{lemma}
\begin{proof}
The inclusion $E_{j+k+2}(i_{k+1,j+k+1}(A_{k+1}))\subset i_{k,j+k}(A_{k})$ is straightfoward. 

Note as $W=V_0\oplus W_0$ with $W_0=\oplus_i V(\omega_i)$, we have $W^{\otimes(k-j)}=(W^{\otimes(k-j)}\otimes V_0)\oplus(W^{\otimes(k-j)}\otimes W_0)$. 
For $i_{k,j+k}(g)=\text{id}_{W^{\otimes j}}\otimes g \in i_{k,j+k}(A_{k})$ with $g\in A_{k}$, we define an element $\bar{g}\in A_{k+1}$ by
\begin{equation*}
\bar{g}=
\begin{bmatrix}
g&0\\
0&0
\end{bmatrix}
\in  \End{((W^{\otimes k}\otimes V_0)\oplus  (W^{\otimes k}\otimes W_0))}
\end{equation*}
with respect to the decomposition of $W^{\otimes k}$ above. 
Then we have $E_{j+k+2}(i_{k+1,j+k+1}(\bar{g}))=E_{j+k+2}(\text{id}_{W^{\otimes j}} \otimes\bar{g})=i_{k,j+k}(g)$. 
\end{proof}
\begin{lemma}
If $k\geq d$, the commuting square is symmetric. 
\end{lemma}
\begin{proof}
By \cite{JS97} Corollary 5.4.4, it suffices to show the inclusion matrices have the following relation:
\begin{center}
$(T_{A_{k}}^{A_{k+1}})^{\rm t} T_{A_{k}}^{A_{j+k}}=T_{A_{k+1}}^{A_{k+j+1}}(T_{A_{j+k}}^{A_{j+k+1}})^{\rm t} $. 
\end{center}
By Lemma \ref{lsymir}, if $k\geq d$, we have $\dim\mathcal{Z}({A_k})=m$.  
So all these inclusion matrices are  $T=T_k$ that we obtained in the proof of Lemma \ref{lsymir}, which is a symmetric one in $\text{Mat}_{m}(\mathbb{Z})$. 
\end{proof}
Now we consider the following towers of $C^*$-algebras:
\begin{equation*}
\begin{aligned}
&A_{j+d}&\subset &~~~~A_{j+d+1}&\subset &~~~~A_{j+d+2}&\subset&\cdots\\
&\cup&   &~~~~~\cup &&~~~~~\cup&&\\
&A_{d}&\subset &~~~~A_{d+1}&\subset &~~~~A_{d+2}&\subset&\cdots
\end{aligned}
\end{equation*}
We have $A_{j+k+1}=\langle A_{j+k},e_{j+k+1}\rangle$. 
As shown before, the unions of these two rows give a pair of $\text{II}_1$ factors $M=M_0\subset M_{j}$. 

\begin{proposition}\label{pcomtt}
With the definition of $A_k,M,M_j$ above, we have  
\begin{center}
$M'\cap M_j\cong A_j$     
\end{center}
for all $j\geq 0$. 
\end{proposition}
\begin{proof}
We have already checked that the first one of the commuting squares above is symmetric and the two rows are the towers obtained from basic constructions. 
By Lemma \ref{lbasicon}, for $k\geq 1$, we know  $i_{d+k+1,j+d+k+1}(A_{d+k+1})$ is equal to $\langle i_{d+k,j+d+k}(A_{d+k}),e_{j+d+k+1}\rangle$ by  $i(e_{d+k+1})=e_{j+d+k+1}$ in Corollary \ref{ciee}. 
By the Ocneanu Compactness theorem (see \cite{JS97} Theorem 5.7.1), we have
\begin{center}
$M'\cap M_j=(i_{d+1,j+d+1}(A_{d+1}))'\cap A_{j+d}$.    
\end{center}
It suffices to show the right-hand side is just $A_j$. 

As shown in Lemma \ref{lbasicon}, we have $A_{k+1}$ always contains all $e_{i}$ with $2\leq i\leq k+1$. 
Within the embedding $i_{k+1,j+k+1}:A_{k+1}\to A_{j+k+1}$, it can be shown that the projections $\{e_{i}\}_{2\leq i\leq k+1}$ are mapped to $\{e_{j+i}\}_{2\leq i\leq k+1}$ respectively. 
So $(i_{d+1,j+d+1}(A_{d+1}))'\cap A_{j+l} \subset \{e_{j+1},\dots,e_{j+d+1}\}'\cap A_{j+l}$. 
Then, by \cite{J83} Proposition 4.1.4, we get $\{e_{j+1},\dots,e_{j+d+1}\}'\cap A_{j+d}=A_j$. 

Moreover, the inclusion $A_j\subset (i_{d+1,j+d+1}(A_{d+1}))'\cap A_{j+d}$ is straightforward. 
Hence we have $M'\cap M_j=A_j$. 
\end{proof}

\underline{The right conditional expectation $E'_{j+k+2}\colon A_{j+k+1}\to A_{j+k}$} There is another conditional expectation $E'_{j+k+2}\colon A_{j+k+1}\to A_{j+k}$ while identifying $A_{j+k}$ as a subalgebra by the inclusion $i_{j+k,j+k+1}\colon f\mapsto \text{id}_W\otimes f$ for $f\in A_{j+k}$.  
(Please note the differences between these $E'_{k}$ and $E_{k}$'s, where $E_k$ comes from the left inclusion $i_k\colon f\mapsto f\otimes \text{id}_W$, see Section \ref{stower}). 
These $E'_{j+k+2}$'s induce a map $E'_{j+1}:\cup_{k\geq 0}A_{j+k+1}\to \cup_{k\geq 0}A_{j+k}$ and further yield a conditional expectation
\begin{center}
$E'_{j+2}:M_{j+1}\to M_j$. 
\end{center}
Let $\xi_j$ be the canonical cyclic trace vector in $L^2(M_j)$. 
By identifying $M_{j+1}$ with the algebra of left action operator on $L^2(M_{j+1})$, $E'_{j+2}$ extends to a projection $e'_{j+2}$ via $e_{j+2}(x\xi_j)=E'_{j+2}(x)\xi_j$.

\begin{corollary}
We have $M_{j+1}=\langle M_j,e'_{j+1}\rangle$ for $j\geq 1$.  
\end{corollary}
\begin{proof}
It follows the fact that $A_{j+k+1}$ is the algebra obtained from the basic construction of the pair $A_{j+k-1}\subset A_{j+k}$ with the conditional expectation $e'_{j+k}$ if $j+k-1\geq d$. 
\end{proof}

We will denote these $e_j^{'}$'s by $e_j$ in the discussion of the infinite-dimensional algebras (factors). 
We may obtain a tower of hyperfinite factors from the basic constructions: 
\begin{center}
$M=M_0\subset M_1\subset^{e_2} M_2\subset^{e_3} M_3\subset \dots$. 
\end{center}
Please note our indices of $e_k$ start from $k=2$, which makes $e_k\in A_k$ and $e_{k}\notin A_{k-1}$.


\section{The Bimodules and Their Fusion Rule}\label{sbimod1}

We first have a review of bimodules over $\text{II}_1$ factors. 
One may refer to \cite{BD97,EK98} for more details. 

Let $A$ and $B$ be $\text{II}_1$ factors.
An {\it $A\text{-}B$ bimodule ${}_{A}{H}_B$} is a pair of commuting normal (unital) representations $\pi_L,\pi_R$ of $A$ and $B^{\text{op}}$ respectively on the Hilbert space $H$.
Here $B^{\text{op}}$ is the opposite algebra of $B$, i.e $b_1\cdot b_2=b_2b_1$, which is also a $\text{II}_1$ factor.
Note that ${}_{A}{H}_B$ is a left $A$-module and right $B$-module with the action denoted as $\pi_L(a)\pi_R(b)\xi=a\cdot \xi \cdot b$ with $a\in A,b\in B,\xi\in H$.
We say ${}_{A}{H}_B$ is {\it bifinite} if the left dimension $\dim_A^L{H}<\infty$ and right dimension $\dim_B^R{H}<\infty$.

\begin{definition}
Let $H,K$ be two $A\text{-}B$ bimodules.
We say $H,K$ are equivalent if we have a unitary $u:H\to K$ such that $u(a\cdot \xi \cdot b)=a\cdot u(\xi) \cdot b$ for all $a\in A,b\in B,\xi\in H$ and denoted by ${}_{A}{H}_B\cong{}_{A}{K}_B$.
Moreover, we denote by
\begin{center}
$\Hom_{A,B}(H,K)=\{T\in B(H,K)|T(a\cdot \xi \cdot b)=a\cdot T(\xi) \cdot b\text{~for all~}a\in A,b\in B,\xi\in H\}$
\end{center}
the space of $A\text{-}B$ intertwiners from $H$ to $K$.
Let $\Hom_{A,B}(H)=\Hom_{A,B}(H,H)$
And we call an $A\text{-}B$ bimodule $H$ irreducible if $\Hom_{A,B}(H)=\mathbb{C}$.
\end{definition}

Note that $\Hom_{A,B}(H)\subset B(H)$ is a von Neumann algebra. 
For a $A$-module $H$, $v\in H$ is called {\it $A$-bounded} if we have a positive constant $c_v$ such that
\begin{center}
$\|xv\|\leq c_v\|x\|_2$ for all $x\in A$,
\end{center}
where $\|x\|_2=\tr(x^*x)^{1/2}$. We write $H^{\text{bdd}}$ for the set of all $A$-bounded vectors in $H$.
It can be shown to be a dense subspace of $H$ and also invariant under the action of $A$ and $A'$ which leads to the following result (see \cite{EK98} and \cite{J08}). 
A proof is also provided below for completeness. 

\begin{lemma}
Assume ${}_{A}{H}_B$ is bifinite, then 
\begin{enumerate}
    \item A vector $v\in H$ is $A$-bounded if and only if it is $B$-bounded;
    \item $\Hom_{A,B}(H)$ is a finite dimensional von Neumann algebra.
\end{enumerate}
\end{lemma}
\begin{proof}
We only prove 2. 
For the proof of $1$, see \cite{J08}. 
Note that $\Hom_{A,B}(H)=A'\cap(B^{\text{op}})'\cap B(H)$ is centainly a von Neumann algebra.
If ${}_{A}{H}_B$ is bifinite, we have
$A\subset (B^{\text{op}})'\cap B(H)$ by the commuting action.
This imlies an inclusion of $\text{II}_1$ factors where
\begin{center}
$[(B^{\text{op}})'\cap B(H):A]=\frac{\dim_{B^{\text{op}}}(H)}{\dim_A(H)}=\frac{1}{\dim_A(H)\dim_B(H)}<\infty$.
\end{center}
Hence $\Hom_{A,B}(H)=A'\cap(B^{\text{op}})'\cap B(H)$ is a relative commutant of a pair of factors with finite index.
So, by \cite{J83}, it is finite-dimensional.
\end{proof}

\begin{corollary}
If ${}_{A}{H}_B$ is bifinite and $p$ is a projection in $\Hom_{A,B}(H)$, then $Hp$ is an irreducible $A\text{-}B$ bimodule if and only if $p$ is minimal.
\end{corollary}
\begin{proof}
If $p$ is minimal, $\Hom_{A,B}(Hp)=p\Hom_{A,B}(H)=\mathbb{C}p\cong \mathbb{C}$.
Otherwise, assume $p=p_1+p_2$ is a decomposition into two subprojections, then $Hp=Hp_1\oplus Hp_2$, which is a direct sum of $A\text{-}B$ bimodules.
\end{proof}

Now let $A,B,C$ be $\text{II}_1$ factors.
Given an $A\text{-}B$ bimodule ${}_{A}{H}_B$ and a $B\text{-}C$ bimodule ${}_{B}{K}_C$, we define the $A\text{-}C$ bimodule of their tensor as \cite{EK98}, which is given by the completion of the algebraic tensor product ${}_{A}{H}^{\text{bdd}}_B\otimes {}_{B}{K}^{\text{bdd}}_C$ of bounded subspace with respect to the inner product defined by
\begin{center}
$\langle v_1\otimes u_1,v_2\otimes v_2\rangle=\langle v_1\langle u_1,u_2 \rangle_B,v_2 \rangle$
\end{center}
Here $\langle u_1,u_2 \rangle_B \in B$ is uniquely determined by
\begin{center}
$\tr(x\langle u_1,u_2 \rangle_B)=\langle xu_1,u_2 \rangle_B$ for all $x\in B$.
\end{center}

It is easy to check the following properties \cite{EK98}:
\begin{enumerate}
\item $\langle \lambda u_1+\mu u_2,u_3 \rangle_B=\lambda \langle u_1,u_3 \rangle_B+\mu \langle u_2,u_3 \rangle_B$,
\item $\langle u_1,u_2 \rangle_B=\langle u_2,u_1 \rangle_B^*$,
\item $\langle xu_1,u_2 \rangle_B=x\langle u_1,u_2 \rangle_B$ and $\langle u_1,xu_2 \rangle_B=\langle u_1,u_2 \rangle_B x^*$.
\end{enumerate}
One may refer to \cite{BD97} for general descriptions of bimodules. 

Consider the tower of $\text{II}_1$ factors
\begin{center}
$M=M_0\subset^{e_1}M_1\subset^{e_2}M_2\subset \cdots$
\end{center}
with $e_k\in M_k$ by iterating the basic constructions $M_{k-1}\subset M_k \subset^{e_{k+1}}M_{k+1}=\langle M_k,e_{k+1}\rangle=(M_k\cup\{e_{k+1}\})''\subset B(L^2(M_k))$. 
Observe the $M\text{-}M$ bimodule $L^2(M_j)$ with the action induced from the two sided action of $A_k$ on $A_{j+k}$ is given by
\begin{center}
$a\cdot \xi \cdot b=i_{k,j+k}(a)\xi (i_{k,j+k}(b^*))$
\end{center}
with $a,b\in A_k, \xi\in A_{k,j+k}$. 
We define a projection 
\begin{center}
$g_k=D^{k(k-1)}(e_{k+1}e_{k}\dots e_2)(e_{k+2}e_{k+1}\dots e_3)\cdots(e_{2k}e_{2k-1}\cdots e_{k+1})$,      
\end{center}
where $D=\sqrt{[M_1:M]}$. 
We have $M\subset M_k\subset^{g_k}M_{2k}$ is the basic construction \cite{BD97}. 
We can also define the actions $\pi_k$ of $M_k,M_{2k}$ on $L^2(M_k)$ as following:
\begin{enumerate}
\item $\pi_k(x)(\hat{z})=\widehat{xz}$, for all $\hat{z}\in \widehat{M_k}\subset L^2(M_k)$, 
\item $\pi_k(xg_ky)(\hat{z})=x\widehat{E^{M_k}_{N}(yz)}$ for all $xg_ky\in M_{2k}$ and $x,y,z\in M_k$. 
\end{enumerate}

\begin{proposition}[\cite{BD97}]\label{p+2}
Let $p,q\in M'\cap M_{2k}$ be two equivalent projections and $M_{2k}\subset^{e_{2k+1}}\subset M_{2k+1}\subset^{e_{2k+2}}M_{2k+2}$. 
Then we have
\begin{equation*}
\begin{aligned}
\pi_k(p)L^{2}(M_k)&\cong \pi_k(q)L^{2}(M_k),  \text{~and}\\
\pi_k(p)L^{2}(M_k)&\cong \pi_{k+1}(pe_{2k+2})L^2(M_{k+1}) 
\end{aligned}
\end{equation*}
as $M\text{-}M$ bimodules. 
\end{proposition}

Let $J_k:L^2(M_k)\to L^2(M_k)$ be the modular conjugation defined by $J_k(\hat{x})=\hat{x^*}$. 
Then we have $J^2_k=\text{id}$ and $J_k\pi_k(M)'J_k=\pi_{k}(M_{2k})$. 

Now we will construct the shifts between the higher commutants. 
Let $\gamma_k:M'\cap M_{2k}\to M'\cap M_{2k}$ be the surjective linear $*$-antiisomorphism defined by $\pi_k(\gamma_k(x))=J_k\pi_k(x)^{*}J_k$. 
Then we get a trace preserving, surjective $*$-isomorphism $sh_{2k}$ given by
\begin{center}
$sh_{2k}=\gamma_{2j+2k}\gamma_{2j}:M'\cap M_{2j}\to M_{2k}'\cap M_{2j+2k}$. 
\end{center}
Then we obtain the following proposition, which generalizes \cite{BD97} Theorem 4.6.c. 

\begin{theorem}\label{ttensor}
Let $p\in M'\cap M_{2j}$, $q\in M'\cap M_{2k}$ be projections and $\sh_{2j}:M'\cap M_{2k}\to M_{2j}'\cap M_{2j+2k}$ be the shift as above. Then, 
\begin{center}
$\pi_{j}(p)L^2(M_j)\otimes \pi_{k}(q)L^2(M_k)\cong \pi_{j+k}(p\sh_{2j}(q))L^2(M_{j+k})$
\end{center}
as $M\text{-}M$ bimodules. 
And $p\sh_{2j}(q)\in M'\cap M_{2j+2k}$ is a projection with trace $\tr_{M_{2j+2k}}(p\sh_{2j}(q))=\tr_{M_{2j}}(p)\tr_{M_{2k}}(q)$. 
\end{theorem}
\begin{proof}
Observe that $p$ and $sh_{2j}(q)$ are commuting projections in $M'\cap M_{2j+2k}$, 
so $psh_{2j}(q)$ is also a projection with the trace as stated above. 

Without loss of generality, we assume $j\geq k$. 
We have $qe_{2k+2}\dots e_{2j}\in M'\cap M_{2j}$. 
By \cite{BD97} Theorem 4.6 c). we obtain
\begin{center}
$\pi_{j}(p)L^2(M_j)\otimes \pi_{j}(qe_{2k+2}\dots e_{2j})L^2(M_k)\cong \pi_{2j}(p\sh_{2j}(qe_{2k+2}\dots e_{2j}))L^2(M_{j+k})$. 
\end{center}
And by Proposition \ref{p+2}, we have $\pi_{j}(qe_{2k+2}\dots e_{2j})L^2(M_j)\cong \pi_{k}(q)L^2(M_k)$. 

We can show that $sh_{2j}(e_i)=e_{2j+i}$. 
Note that $q$ commutes with all $e_{2k+2},\dots,e_{2_j}$, we have
$p\sh_{2j}(qe_{2k+2}\dots e_{2j})=p\sh_{2j}(q)e_{2j+2k+2}\dots e_{4j}$. Then by Proposition \ref{p+2} again, we obtain
$\pi_{2j}(p\sh_{2j}(qe_{2k+2}\dots e_{2j}))L^2(M_{j+k})\cong \pi_{j+k}(p\sh_{2j}(q))L^2(M_{j+k})$, 
which completes the proof. 
\end{proof}

\underline{The construction of bimodules}
Let $M_j$'s be the $\text{II}_1$ factors that are constructed in Section \ref{stower} and Section \ref{sfactor}. 
Let us consider the Jones tower of $\text{II}_1$ factors
\begin{center}
$M=M_0\subset M_1\subset^{e_2}M_2\subset \cdots$
\end{center}
By Proposition \ref{pcomtt} and \cite{BD97} Proposition 3.2, we have $\Hom_{M\text{-}M}({}_{M}{L^2(M_j)}_n)=M'\cap M_{2j}\cong A_{2j}$. 

Recall each $V_i$ must be in a summand of $W^{\otimes k}$ when $k\geq d(l)$ by Lemma \ref{liffdl}. 
For each simple object $V_i$ in $R_l(\mathfrak{g})$, we define
\begin{center}
$k_i=\min\{k\geq 0|\Hom(V_i,W^{\otimes k})\neq 0\}$, 
\end{center}
which is the minimal integer $k$ such that $W^{\otimes k}$ contains $V_i$. 
Note if $V_i$ is fundamental, $k_i=1$.

Define a map $\phi:\{1,\dots,m\}\to \mathbb{Z}_2$ by $\phi(i)=k_i\mod 2$. 
It should be mentioned that $\phi(1)=0$ as $V_1=W^{\otimes 0}$  and $\phi(i)=1$ if $V_i$ is fundamental as $W$ is the direct sum of fundamental ones. 
We are now able to construct the simple bimodules as follows. 
For any central projection $p\in A_k$, we let $z(p)$ denote the projection in $\mathcal{Z}(A_k)$ which is equivalent to $p$ in $A_k$. 
As all these $A_k$'s are multi-matrix algebras, $z(p)$ would be a sum of diagonal matrices with only $0$ and $1$ on the diagonals. 
\begin{itemize}
\item If $\phi(i)=0$, i.e. $k_i$ is even, say $k_i=2r_i$. 
We take a minimal projection $g_i$ in $A_{2r_i}=M'\cap M_{2r_i}$ such that $z(g_i)$ is the projection from $W^{\otimes 2r_i}$ on $V_i$. 
We let $H_i=\pi_{r_i}(g_i)L^2(M_{r_i})$.  
    
\item If $\phi(i)=1$, i.e. $k_i$ is odd, say $k_i=2{r_i}-1$. 
We take a minimal projection $g_i'\in A_{2{r_i}-1}=M'\cap M_{2{r_i}-1}$ such that $z(g_i)$ is also the projection from $W^{\otimes 2r_{i}-1}$ on $V_i$. 
Define $g_i=g_i'\otimes \text{id}_{V_0}\in A_{{k_i}+1}=M'\cap M_{2r_i}$ and let $H_i=\pi_{r_i}(g_i)L^2(M_{r_i})$.
\end{itemize}
By Proposition \ref{p+2}, these bimodules only depend on the equivalence class of the projections but not the particular choice of the minimal projection $g_i$.
In particular, we let $H_1$ denote the standard bimodule $L^2(M)$, which corresponds to the unique nontrivial projection in $\End(W^{\otimes 0})\cong\mathbb{C}$.

\underline{The construction of the fusion category $\text{Bimod}(M,M_1)$} 
Define a category
\begin{center}
    $\text{Bimod}(M,M_1)=\{$the equivalence classes of $M$-$M$ bimodules in $\cup_{j}L^2(M_j)\}$, 
\end{center} 
where $M_j$ is obtained from the basic construction of $M_{j-2}\subset M_{j-1}$ for each $j\geq 2$. 
It is well-known to be the tensor category generated by the equivalence class $\pi_{j}(p)L^2(M_j)$ with a minimal projection $p\in M'\cap M_{2j}$ for $j\geq 0$. 
It can also be shown $\text{Bimod}(M,M_1)$ is generated by the fundamental ones: $H_i=\pi_1(p_i)L^2(M_1)$ with the projection $p_i\colon W=V(0)\oplus(\oplus_{1\leq k\leq n}V(\omega_k))\to V(\omega_i)$. 

\begin{lemma}
$\text{Bimod}(M,M_1)$ is a fusion category with simple objects $H_i$'s defined above.
\end{lemma}
\begin{proof}
By Theorem \ref{ttensor}, they are closed under tensor products. 
Since the inclusion $M\subset M_1$ of $\text{II}_1$ factors is of the finite depth $d(l)$, there are finitely many simple objects. 
These objects are in one-to-one correspondence with the (equivalence classes of) minimal projections in the higher commutants $M'\cap M_{2i}$ \cite{BD97},
which give us the bimodule $H_i$'s.  
\end{proof}

The rest of this section is mainly devoted to proving the following theorem. 

\begin{theorem}\label{tfuscat}
As a fusion category, $\text{Bimod}(M,M_1)\cong R_l(G)$ .  
\end{theorem}
The proof is based on several statements below. 
\begin{lemma}\label{lshif}
Take any $f\in A_{2k}=M\cap M_{2k}$, we have $\sh_{2j}(f)=i_{2j,2j+2k}(f)$.
\end{lemma}
\begin{proof}
Observe that $i_{2j,2j+2k}(A_{2k})\subset A_{2j+2k}$ and it commutes with $A_{2j}$, we have $i_{2j,2j+2k}(A_{2k})\subset M_{2j}'\cap M_{2j+2k}$ which can be further shown to be a surjective, trace preserving, $*$-isomorphism. 
Then the proof reduces to the construction of the isomorphism $sh_{2j}$.  
\end{proof}


\underline{A functor $\Psi\colon \text{Bimod}(M,M_1)\to R_l(G)$} is defined as follows:
\begin{center}
$\Psi(\pi_{j}(p)L^2(M_j))=p(W^{\otimes 2j})$.
\end{center}
where $p\in M\cap M_{2j}$ for some $j$. 

\begin{lemma}
We have $\Psi(H_i)=V_i$ for all $1\leq i\leq n$. 
\end{lemma}
\begin{proof}
If $\phi(i)=0$, it is straightforward by the construction of $H_i$'s above. 
If $\phi(i)=1$, we have $\Psi(H_i)=V_i\otimes \text{id}_{V_0}(W)=V_i\otimes V_0=V_i$ by Lemma \ref{lshif}. 
\end{proof}

\begin{lemma}
$\Psi(\pi_{j}(p)L^2(M_j))$ depends only on the isomophism class of $p$. 
Hence $\Psi$ is well-defined.  
\end{lemma}
\begin{proof}
Let $p$ be (equivalent to) a minimal projection in the $i$-th simple summand of $A_{2j}$.  
Assume there is another projection $p'\in M\cap M_{2j'}$ which is equivalent to $p$. 
Then it is also equivalent to a minimal projection in the $i$-th simple summand of $A_{2j'}$.   
Assume $j\geq j'$, then $p$ is equivalent to $p'e_{2j'+2}\dots e_{2j}$ in $M'\cap M_{2j}$.  
We have $\pi_{j}(p)L^2(M_j)\cong \pi_{j'}(p')L^2(M_j')$ and both of them are mapped to $V_i$ under the functor $\Psi$. 
\end{proof}

Now it is clear that $\Psi^{-1}(V_k)$ is the equivalence class of the minimal projections in the $k$-th simple summand.

\begin{proposition}
The functor $\Psi$ preserves tensor products.
\end{proposition}
\begin{proof}
Take any two projections $p\in M_{2j},q\in M_{2k}$. 
By Theorem \ref{ttensor}, we have $\pi_{j}(p)L^2(M_j)\otimes \pi_{k}(q)L^2(M_k)\cong \pi_{j+k}(p\sh_{2j}(q))L^2(M_{j+k})$. 
On the other hand, $p\sh_{2j}(q)=p\cdot i_{2j,2j+2k}(q)=p\cdot (\text{id}_{W^{2j}}\otimes q)$. 
Hence $p\sh_{2j}(q)(W^{\otimes 2j+2k})=p\cdot (\text{id}_{W^{2j}}\otimes q)(W^{\otimes 2j+2k})=p(W^{\otimes 2j})\otimes q(W^{\otimes 2k})$, which completes the proof. 
\end{proof}

\begin{lemma}
The functor $\Psi$ preserves direct sums. 
\end{lemma}
\begin{proof}
Now we take two irreducible bimodules $H_j,H_k$ so that $\Psi(H_j)=V_j,\Psi(H_k)=V_k$. 
Let us consider $H_j\oplus H_k$ which is $\pi_{r_j}(g_j)L^2(M_{r_j})\oplus \pi_{r_k}(g_k)L^2(M_{r_k})$. 
Assume $j\geq k$, by Proposition \ref{p+2}, we have this is also the bimodule $\pi_{r_j}(g_j\oplus g_{k}e_{2k+2}\dots e_{2j})L^2(M_{r_j})$ which is a direct sum of two bimodules. 
For the first one, $\Psi(\pi_{r_j}(g_j)L^2(M_{r_j})=V_j$ is clear. 
And by Proposition \ref{p+2} again, we have $\Psi(\pi_{r_j}( g_{k}e_{2k+2}\dots e_{2r_j}))=\Psi(\pi_{r_k}(g_k)L^2(M_{r_k}))=\Psi(H_k)=V_k$. 
Hence $\Psi(H_j\oplus H_k)=\Psi(H_j)\oplus \Psi(H_k)$. 
\end{proof}

\begin{proof}[Proof of Theorem \ref{tfuscat}]
Take any two irreducible representations $V_i,V_j$ of $LG$.
Assume we have the following decomposition of their tensor product:
\begin{center}
$V_i\otimes V_j=\oplus_{k=0}^m m_{i,j}^{k}\cdot V_k$, $m_{i,j}^{k}\in \mathbb{Z}_{\geq 0}$. 
\end{center}
We want to show that $H_i\otimes H_j$ has the same decomposition into the irreducible $M$-$M$ bimodules $H_k$'s. 

By Theorem \ref{ttensor}, note $g_i\in M'\cap M_{2r_i}$ and $g_j\in M'\cap M_{2r_j}$, we obtain
\begin{center}
$\pi_{r_i}(g_i)L^2(M_{r_i})\otimes \pi_{r_j}(g_j)L^2(M_{r_j})\cong \pi_{r_i+r_j}(g_i\sh_{2r_i}(g_j))L^2(M_{r_i+r_j})$,    
\end{center}
where $g_i\sh_{2r_i}(g_j)$ is a projection $\in M'\cap M_{2r_i+2r_j}=A_{2r_i+2r_j}$ and $g_i$ commutes with the minimal projection $sh_{2r_i}(g_j)\in M_{2r_i}'\cap M_{2r_i+2r_j}$. 
We then have $\Psi(\pi_{r_i+r_j}(g_i\sh_{2r_i}(g_j))L^2(M_{r_i+r_j})=V_i\otimes V_j$ by the fact that
\begin{center}
    $z(g_i)z(\sh_{2r_i}(g_j))(W^{\otimes 2r_i+2r_j})=z(g_i)(W^{\otimes 2r_i})\otimes z(g_j)(W^{\otimes 2r_j})$. 
\end{center}
By taking $\Phi^{-1}$, we obtain $H_i\otimes H_j=\oplus_{k=0}^m m_{i,j}^{k}\cdot H_k$. 
\end{proof}



\bibliographystyle{abbrv}
\typeout{}
\bibliography{MyLibrary} 



\end{document}